\documentclass[12pt]{amsart}
\usepackage{mathptmx}
\usepackage{amsmath,amsfonts,amssymb}
\usepackage{mathtools}
\usepackage{mathrsfs}
\DeclareMathAlphabet{\mathscrbf}{OMS}{mdugm}{b}{n}
\usepackage[english]{babel}
\usepackage{epsfig}
\usepackage[all]{xy}
\usepackage{graphicx}
\usepackage{latexsym}
\usepackage{verbatim}
\usepackage{enumerate}
\usepackage{enumitem}
\usepackage[svgnames]{xcolor}
\usepackage{ulem}
\newtagform{orange}{\color{orange}(}{)}
\usepackage{cancel}
\usepackage{tabularx}
\usepackage{caption}
\def\HYPER{\relax}

\ifdefined\HYPER
\usepackage[hidelinks, pdfstartview=FitB, pdfpagemode=UseNone]{hyperref}
\else
\renewcommand{\href}[2]{\relax}
\renewcommand{\url}[1]{#1}
\fi

\allowdisplaybreaks
\mathtoolsset{showonlyrefs,showmanualtags}
\usepackage{thmtools,thm-restate}

\newcommand{\D}{\mathbb D}

\newcommand{\C}{\mathbb C}
\newcommand{\N}{\mathbb N}

\newcommand{\Aut}{{\sf Aut}}
\renewcommand{\H}{{\mathcal{H}}}

\def\eps{\varepsilon}

\def\dD{\mathop{{\rm d}_\D}}

\newcommand{\mcite}[1]{\csname b@#1\endcsname}
\newtheoremstyle{break}
{8pt}{8pt}%
{\itshape}{}%
{\bfseries}{}%
{\newline}{}%
\theoremstyle{break}

\theoremstyle{break}
\newtheorem{theorem}{Theorem}[section]
\newtheorem*{theorem*}{Theorem}
\newtheorem{lemma}[theorem]{Lemma}
\newtheorem*{lemma*}{Lemma}
\newtheorem{proposition}[theorem]{Proposition}

\theoremstyle{break}
\newtheorem{theoremliterature}{Theorem}

\theoremstyle{break}

\theoremstyle{remark}
\newtheorem{remark}{Remark}[section]

\numberwithin{equation}{section}

\definecolor{dgreen}{rgb}{0,0.5,0}

\newcommand{\hide}[1]{}

\author[D.~Kraus]{Daniela Kraus}
\address{D. Kraus: Department of Mathematics, University of W\"urzburg, Emil Fischer Strasse 40, 97074, W\"urzburg, Germany.} \email{daniela.kraus@uni-wuerzburg.de}

\author[A.~Moucha]{Annika Moucha$^\dag$}
\address{A. Moucha: Department of Mathematics, University of W\"urzburg, Emil Fischer Strasse 40, 97074, W\"urzburg, Germany.} \email{annika.moucha@uni-wuerzburg.de}

\author[O.~Roth]{Oliver Roth}
\address{O. Roth: Department of Mathematics, University of W\"urzburg, Emil Fischer Strasse 40, 97074, W\"urzburg, Germany.} \email{oliver.roth@uni-wuerzburg.de}

\thanks{$^\dag\,$Supported by the Alexander von Humboldt Stiftung}

\title{Stability of Blaschke products under forward iteration}

\begin{document}

\begin{abstract}
  Forward iteration of holomorphic self-maps generalizes the iteration of a
  single function in a natural way. This framework arises in complex dynamics,
  for instance in the study of wandering domains and in seeking suitable
  extensions of the Denjoy--Wolff theorem. Here, we consider forward iteration of Blaschke products. We prove that 
  the classes of indestructible and maximal Blaschke products are stable under forward iteration.
\end{abstract}
	
    \maketitle

\section{Introduction}
We identify two classes of Blaschke products that are stable under forward iteration and settle two open problems posed in \cite{ferreiraNoteForwardIteration2023} and \cite{krausMaximalBlaschkeProducts2013}.

 \medskip
 
 Let $\D := \{ z \in \C : |z|<1 \}$ denote the open unit disk, and let $\mathcal H$ be the set of holomorphic self-maps of $\D$. The class $\mathcal H$ is closed under composition. For a sequence $(f_n) \subseteq \mathcal H$, the
 \emph{forward  iterates} of $(f_n)$ are defined by
\[
	F_n := f_n \circ \dots \circ f_1 .
\]
While the dynamics of iterating a single function $f \in \mathcal H$ are well understood through the Denjoy--Wolff theorem, the behavior of the forward iterates $(F_n)$ of a sequence $(f_n)$ is considerably more subtle. This has motivated numerous works on forward iteration and related dynamical systems; see, for instance, \cite{abateRandomIterationHyperbolic2022a,abate2025iteratedfunctionsystemsholomorphic,beniniClassifyingSimplyConnected2022,christodoulouNoteBoundaryDynamics2025a,christodoulouStabilityDenjoyWolffTheorem2021a,ferreiraNoteForwardIteration2023,FerreiraNicolau2025,GouezelKarlsson2020,JacquesShort2022}.

\medskip

Suppose the forward iterates of a sequence $(f_n) \subseteq \mathcal H$ converge pointwise (equivalently, uniformly on compact subsets of $\D$) to a limit $F \in \mathcal H$. The following theorem characterizes when $F$ is nonconstant. We recall that the hyperbolic distortion of a function $f \in \mathcal H$ is defined by
\[
	D_h f(z) := \frac{(1-|z|^2)\,|f'(z)|}{1-|f(z)|^2}, \quad z \in \D.
\]

\begin{theoremliterature}[{Benini--Evdoridou--Fagella--Rippon--Stallard~\cite[Th.~2.1]{beniniClassifyingSimplyConnected2022}}; Abate--Short~{\cite[Th.~E]{abate2025iteratedfunctionsystemsholomorphic}}]\label{thm:BeniniEtAl}
	Let $(f_n)$ be a sequence in $\mathcal H$. Suppose that each $f_n$ is nonconstant and that the forward iterates $F_n = f_n \circ \dots \circ f_1$ converge to a limit $F \in \mathcal H$. Then $F$ is nonconstant if and only if
	\begin{equation}\label{eq:IterationNonConstantCondition}
		\sum_{n \ge 1} \big(1 - D_hf_n(a)\big) < \infty, \qquad 
			\end{equation}
	for one -- and hence all -- $a \in \D$.
\end{theoremliterature}

\begin{remark}\label{rem:GeneralizationOfForwardResult}
In \cite[Th.~2.1]{beniniClassifyingSimplyConnected2022}, the result is established in the normalized case $f_n(0)=0$ for all $n \in \N$. The general case, as well as an extension to maps on arbitrary hyperbolic Riemann surfaces, was obtained in \cite[Th.~E]{abate2025iteratedfunctionsystemsholomorphic}.  
For the unit disk, the general case can also be deduced from the normalized one by setting

\begin{equation} \label{eq:Normalization}
\widetilde{f_n} := T_{F_n(a)} \circ f_n \circ T_{F_{n-1}(a)},
\end{equation}
where $F_0 := \mathrm{id}$ and 
\[
T_w(z) \textcolor{blue}{:}= \frac{w-z}{1-\overline{w}z}, \qquad w \in \D\,, 
\]
denotes the involutive automorphism of $\D$ with $T_w(0) = w$.  
For completeness, a short proof of Theorem \ref{thm:BeniniEtAl}, based  on the infinitesimal form of the Schwarz--Pick lemma, is given in Section~\ref{sec:TheoremA}.
\end{remark}

Theorem \ref{thm:BeniniEtAl} motivates the study of properties of holomorphic
self-maps that are preserved under forward iteration. We call a subset
$\mathcal{M} \subseteq \H$  \emph{stable under forward iteration} if, for any
sequence in~$\mathcal{M}$ whose forward iterates converge to a nonconstant
function $F \in \mathcal{H}$, the limit $F$ also lies
in~$\mathcal{M}$. Suppose $\mathcal{M}$ contains the identity
  map.  Then $\mathcal{M}$ can only be stable under forward iteration if it is a (composition) semigroup,  but this condition is not sufficient. Given their ubiquity in the theory of bounded holomorphic functions, stability under forward iteration is of particular interest for inner functions and Blaschke products. 
      Recall that $f \in \H$ is an \emph{inner function} if its radial limits exist and are unimodular almost everywhere on $\partial\D$ with respect to Lebesgue measure.
         A \emph{Blaschke product} is an inner function $B$ of the form\footnote{We adopt the convention $|z_n|/z_n := 1$ whenever $z_n = 0$.}
\begin{equation}
	B(z) = \eta \prod_{n=1}^N \frac{|z_n|}{z_n} \frac{z_n - z}{1 - \overline{z_n} z},
\end{equation}
where $\eta \in \partial \D$, $N \in \N \cup \{\infty\}$, and the sequence $(z_n) \subseteq \D$ satisfies the Blaschke condition
\begin{equation}
	\sum_{n=1}^N (1 - |z_n|) < \infty.
\end{equation}

It was proved in \cite{Ryff1966} that inner functions form a composition
semigroup and G. Ferreira \cite{ferreiraNoteForwardIteration2023}
 recently established their stability under forward iteration.

\medskip

Forward iteration of Blaschke products requires additional considerations. Frostman's theorem~\cite[Ch.~II, Th.~6.4]{garnettBoundedAnalyticFunctions2007} states that every inner function can be written as ${T \circ B}$, where $T$
belongs to the group $\Aut(\D)$ of conformal automorphisms  of the unit disk and $B$ is a Blaschke product. 
This motivates the class of \emph{indestructible Blaschke products}: a Blaschke product $B$ is indestructible if $T \circ B$ is again a Blaschke product for every $T \in \Aut(\D)$. Indestructible Blaschke products first appeared in the work of Heins \cite{Heins1953,Heins1955}, and their systematic study was initiated in \cite{mclaughlinExceptionalSetsInner1972}. Subsequent work has clarified various structural properties;  see e.g.~\cite{Akeroyd2016,Akeroyd2022,Bishop1993,krausCompositionDecompositionIndestructible2013, Morse1980,rossIndestructibleBlaschkeProducts2008}.

\medskip

Regarding forward iteration of finite Blaschke products,  Ferreira ~\cite{ferreiraNoteForwardIteration2023} established the following.

        	\begin{theoremliterature}[Ferreira \cite{ferreiraNoteForwardIteration2023}] \label{thm:Ferreira}
          Let $(b_n)$ be a sequence of finite Blaschke products and suppose that the forward iterates ${B_n:=b_n \circ \cdots \circ b_1}$ converge to a nonconstant limit $B \in \mathcal{H}$. Then $B$ is a Blaschke product.
          	\end{theoremliterature}

                It is instructive to compare Theorem \ref{thm:Ferreira} with the classical result of Carath\'eodory \cite[p.~13-14]{CaratheodoryII} that every function in $\mathcal{H}$ is the locally uniform limit of a sequence of finite Blaschke products. 
This indicates that forward iteration preserves structural properties more effectively than general convergence.

\medskip

Although Theorem~\ref{thm:Ferreira} assumes finite Blaschke products, it concludes that the limit is a (possibly infinite) Blaschke product, leading naturally to the following considerations. Noting that the composition of two indestructible Blaschke products is again indestructible \cite[Theorem 1.1]{krausCompositionDecompositionIndestructible2013}, Ferreira \cite[p.~8]{ferreiraNoteForwardIteration2023} observed that the proof of Theorem~\ref{thm:Ferreira} remains valid under the weaker assumption that $(b_n)$ is a sequence of indestructible Blaschke products. He then posed the question of whether the limit is itself indestructible. Our first result answers this question affirmatively and furthermore establishes the  converse.

          	\begin{restatable}{theorem}{IterationIBP}
	\label{thm:IterationOfIBP}
		Let $(b_n)$ be a sequence in $\H$ and suppose the forward iterates $B_n:=b_n \circ \cdots \circ b_1$ converge to a nonconstant limit $B \in \mathcal{H}$.  Then the following are equivalent:
		\begin{enumerate}[label={\rm (\alph*)}, ref={\hbox{(\alph*)}}, left=0em]
			\item \label{item:IBPsequence} Each  $b_n$ is an indestructible Blaschke product.
			\item \label{item:IBPlimit} $B$ is an indestructible Blaschke product.
		\end{enumerate}
              \end{restatable}

Theorem \ref{thm:IterationOfIBP} \ref{item:IBPsequence} $\Longrightarrow$ \ref{item:IBPlimit} improves Theorem \ref{thm:Ferreira}, yielding a stronger conclusion under a weaker hypothesis. Our proof relies on the stability of inner functions under forward iteration \cite{ferreiraNoteForwardIteration2023}.

\medskip

Next, we consider whether the conclusion of Theorem~\ref{thm:Ferreira} can be strengthened under its original assumption that $(b_n)$ is a sequence of finite Blaschke products. This leads to the class of \emph{maximal Blaschke products}, which arise naturally in hyperbolic geometry and are closely linked to the Schwarz--Pick lemma. 
The idea originates in the work of Ahlfors \cite{Ahlfors1938}, Nehari \cite{Nehari1947}, and Heins \cite{Heins1962}. The crucial result -- that extremal functions in the Ahlfors--Nehari extensions of the Schwarz--Pick lemma are always Blaschke products -- was established in \cite{krausCriticalSetsBounded2013}. For recent work, see \cite{bracciNewSchwarzPickLemma2023,ivriiPrescribingInnerParts2019,ivriiCriticalStructuresInner2021,ivriiAnalyticMappingsUnit2024,ivriiAnalyticMappingsUnit2025,krausCriticalSetsBounded2013,krausCriticalPointsGauss2013,krausMaximalBlaschkeProducts2013}.
The precise definition and a short discussion of maximal Blaschke products are given in Section~\ref{sec:proofMBP}. Here, we merely note that
\[
\{\text{finite Blaschke products}\} \subsetneq \{\text{maximal Blaschke products}\} \subsetneq \{\text{indestructible Blaschke products}\},
\]
and that the composition of two maximal Blaschke products is again maximal, see \cite[Theorem 1.7]{krausMaximalBlaschkeProducts2013} and Theorem \ref{thm:MBPdecomposition} below. A~natural question is whether they remain stable under forward iteration.

\begin{restatable}{theorem}{IterationMBP}
	\label{thm:IterationOfMBP}	Let $(b_n)$ be a sequence in $\H$ and suppose the forward iterates $B_n:=b_n \circ \cdots \circ b_1$ converge to a nonconstant limit $B \in \mathcal{H}$.  Then the following are equivalent:
		\begin{enumerate}[label={\rm (\alph*)}, ref={\hbox{(\alph*)}}, left=0em]
		\item \label{item:MBPsequence} Each  $b_n$ is a maximal Blaschke product.
                \item \label{item:MBPlimit} $B$ is a maximal Blaschke product.
		\end{enumerate}
	\end{restatable}

        Since finite Blaschke products are maximal, the implication \ref{item:MBPsequence} $\Longrightarrow$ \ref{item:MBPlimit} of Theorem \ref{thm:IterationOfMBP} shows that every nonconstant limit of  forward iterates of finite Blaschke products is a maximal Blaschke product. This strengthens Theorem \ref{thm:Ferreira} (which ensured such limits are `only' Blaschke products) and refines Theorem \ref{thm:IterationOfIBP} (which established their indestructibility).
     The proof of Theorem~\ref{thm:IterationOfMBP} is entirely different from that of Theorem~\ref{thm:IterationOfIBP} and does not  rely on the stability of inner functions under forward iteration.

A key step in proving the implication \ref{item:MBPlimit} $\Longrightarrow$ \ref{item:MBPsequence} of Theorem \ref{thm:IterationOfMBP} relies on the following result, which is of independent interest. 

\begin{theorem}\label{thm:MBPdecomposition}
  Let $A_1,A_2\in\mathcal{H}$ and let $A:=A_1\circ A_2$ . Then the following are equivalent:
  \begin{enumerate}[label={\rm (\alph*)}, ref={\hbox{(\alph*)}}, left=0em]
  \item \label{item:MBPComposition} $A_1$ and $A_2$ are maximal Blaschke products.
  \item \label{item:MBPComposite} $A$ is a maximal Blaschke product.
		\end{enumerate}
                
\end{theorem}

\begin{remark} \label{rem:MBPdecomposition}
Theorem \ref{thm:MBPdecomposition} is not entirely new. As noted above, the implication \ref{item:MBPComposition} $\Longrightarrow$ \ref{item:MBPComposite} is Theorem 1.7 in \cite{krausMaximalBlaschkeProducts2013}.  Proposition 5.2 in the same reference asserts that if $A_1 \circ A_2$ is a maximal Blaschke product, then $A_1$ is maximal. The question of whether $A_2$ is also maximal was posed in \cite[Probl.~5.3]{krausMaximalBlaschkeProducts2013} and is answered affirmatively here.
\end{remark}

             The paper is organized as follows. Section~\ref{sec:propertiesofforwarditeration} discusses general properties of forward iteration.
             Theorem~\ref{thm:IterationOfIBP} is proved in Section~\ref{sec:proofIBP}, while Theorems~\ref{thm:IterationOfMBP} and \ref{thm:MBPdecomposition} are proved in Section~\ref{sec:proofMBP}.        
             The proof of Theorem~\ref{thm:IterationOfMBP}
               rests on the refinement of the Schwarz--Pick lemma provided by maximal Blaschke products, see Section \ref{sec:proofMBP}, in particular \eqref{eq:SchwarzPickMBPDef}.  As a further illustration of this approach, we apply the classical Schwarz--Pick lemma to present a short proof of Theorem~\ref{thm:BeniniEtAl} in Section~\ref{sec:TheoremA}.

	\section{Properties of forward iteration}\label{sec:propertiesofforwarditeration}

For nonconstant $g \in \mathcal{H}$ and $c \in \mathbb{D}$, let $Z_g(c)$ denote the multiset\footnote{%
A multiset is a collection in which elements occur with multiplicity; see~\cite{Blizard1989} and~\cite[p.~473]{Knuth1997}.%
} 
of zeros of $g - g(c)$, counted with multiplicities. The notation $\prod_{z \in Z_g(c)} |z|$ refers to the corresponding (possibly infinite) product, with each zero repeated according to its multiplicity.

\medskip

We begin by examining the relationship between the images of the limit function $F$ and those of the forward iterates $F_n$ of a sequence $(f_n)$ in $\mathcal{H}$ at a given point $c \in \mathbb{D}$.

\begin{lemma}\label{lem:ZeroSetInclusion}
Let $(f_n)$ be a sequence in $\mathcal{H}$, and suppose the forward iterates $F_n := f_n \circ \dots \circ f_1$ converge to a nonconstant limit $F \in \mathcal{H}$. Let $c \in \mathbb{D}$, and set $Z = Z_F(c)$ as well as $Z_n = Z_{F_n}(c)$ for $n \in \mathbb{N}$. Then
\begin{subequations}
\begin{equation}\label{eq:ZeroSetInclusion}
Z_n \subseteq Z_{n+1} \quad \text{for all } n \in \mathbb{N}, 
\qquad 
Z = \bigcup_{n \ge 1} Z_n,
\end{equation}
and
\begin{equation}\label{eq:ProductLimit}
\prod_{z \in Z\setminus \{0\}} |z| = \lim_{n \to \infty} \prod_{z \in Z_n\setminus\{0\}} |z|.
\end{equation}
\end{subequations}
\end{lemma}

\begin{proof}

Only~\eqref{eq:ProductLimit} requires proof, since~\eqref{eq:ZeroSetInclusion} follows as in~\cite[Lem.~3.1]{ferreiraNoteForwardIteration2023}.

\medskip

 If $Z$ is finite, then $Z=Z_N$ for all sufficiently large $N \in \N$ and the equality in~\eqref{eq:ProductLimit} is immediate.   Otherwise, each $Z_n$, as well as $Z$,  satisfies the  Blaschke condition,  so the products in  \eqref{eq:ProductLimit} converge.
Enumerate $Z\setminus\{0\} = (c_j)$.   By~\eqref{eq:ZeroSetInclusion}, for each $k$ there exists $N(k)$ such that 
$\{c_1, \dots, c_k\} \subseteq Z_{N(k)}\setminus\{0\}$. Then
$$ \prod_{j=1}^k |c_j| \ge \prod_{z \in Z_{N(k)}\setminus\{0\}} |z| \ge  \prod_{z \in Z\setminus\{0\}} |z| \, .$$
This implies
$$\prod_{z \in Z\setminus\{0\}} |z| =\lim \limits_{k \to \infty} \prod_{z \in Z_{N(k)}\setminus\{0\}} |z|$$
and \eqref{eq:ProductLimit} follows because $\prod \limits_{z \in Z_n\setminus\{0\}} |z|$ is monotonic. 
\end{proof}

\noindent
We denote the $k$--th coefficient in the power series expansion of a holomorphic function 
$f : \mathbb{D} \to \mathbb{C}$ about the origin by $\widehat{f}(k)$.

\begin{lemma} \label{lem:OrderOfZeroEventuallyConstant}
Let $(f_n)$ be a sequence in $\mathcal{H}$, and suppose the forward iterates $F_n := f_n \circ \dots \circ f_1$ 
converge to~a nonconstant limit $F \in \mathcal{H}$. 
Assume moreover that $F_n - F_n(0)$ has a zero of order $K_n$ at $0$ for each~$n$, 
and that $F - F(0)$ has a zero of order $K$ at $0$. 
Then the sequence $(K_n)$ is nondecreasing and eventually constant with value $K$.
\end{lemma}

\begin{proof}
We first assume that $f_n(0) = 0$ for all $n \in \mathbb{N}$. Then also $F_n(0) = F(0) = 0$. 
Let $j_n \in \mathbb{N}$ denote the order of the zero of $f_n$ at $0$, then
\[
K_n = j_n \cdot K_{n-1},
\]
so the sequence $(K_n)$ is nondecreasing.  
Since $\widehat{F}(K) \neq 0$, there exists $N \in \N$ such that $\widehat{F_n}(K) \neq 0$ for each $n \ge N$. This implies $K_n \le K$ for each $n \ge N$, and shows that $(K_n)$ is eventually constant with value $L \le K$. By increasing $N$ if necessary, we may assume that $K_n=L$ for each $n \ge N$.

\medskip

By Theorem~\ref{thm:BeniniEtAl}, there exists $\varepsilon > 0$ such that
\[
\prod_{j \ge n} |f_j'(0)| \ge \varepsilon \quad \text{for each } n \ge N.
\]
Note that 
\[
\widehat{F_n}(L) = \prod_{j=N+1}^n f_j'(0) \cdot \widehat{F_N}(L), \qquad n \ge N,
\]
since $\widehat{F_N}(j)=0$ for $j=0,\ldots, L-1$. Consequently
\[
|\widehat{F}(L)| = \lim_{n \to \infty} |\widehat{F_n}(L)| \ge \varepsilon |\widehat{F_N}(L)| > 0,
\]
thus $L \ge K$.
For the general case, in which the functions $f_n$ do not necessarily fix the origin, we apply the previous argument to
\begin{equation} \label{eq:normalizing}
\widetilde{f_n} := T_{F_n(0)} \circ f_n \circ T_{F_{n-1}(0)},
\end{equation}
where we set $F_0 := \mathrm{id}$ and
\[
T_w(z) := \frac{w-z}{1-\overline{w} z}, \quad w  \in \mathbb{D}.
\]
Note that the order of the zero at the origin of $F_n - F_n(0)$ then coincides with that of 
$\widetilde{F_n} := \widetilde{f_n}~\circ~\dots~\circ~\widetilde{f_1}$, and  $\widetilde{F_n} \to T_{F(0)} \circ F \circ T_0$.
\end{proof}

Next, we show that surjective functions in $\mathcal{H}$ are stable under forward iteration.

\begin{proposition}\label{prop:Surjective}
  Let $(f_n)$ be a sequence in $\mathcal{H}$, and suppose the forward iterates $F_n := f_n \circ \dots \circ f_1$ 
converge to a nonconstant limit $F \in \mathcal{H}$. 
If $f_n : \D \to \D$ is onto for each $n \in \mathbb{N}$, then the limit $F: \D \to \D$ is onto.
\end{proposition}

The following remark, originating in~\cite{beniniClassifyingSimplyConnected2022,ferreiraNoteForwardIteration2023}, will play a  role in the proof of Proposition~\ref{prop:Surjective} and of Theorem~\ref{thm:IterationOfIBP}. 
We state it here for later reference.

\begin{remark}\label{rem:factorization}
Let $F \in \mathcal{H}$ be the nonconstant limit function of the forward iterates of a sequence $(f_n)$ in $\mathcal{H}$.
Then, for each $N \in \mathbb{N}$, there exists a nonconstant $H_N \in \mathcal{H}$ such that $F = H_N \circ F_N$.  
Indeed, fixing $N \in \mathbb{N}$ and defining
\[
h_{N,n} := f_{N+n} \circ \cdots \circ f_{N+1}, \qquad n \in \mathbb{N},
\]
we observe that $(h_{N,n})$ forms a normal family.  
Since $F_N$ is nonconstant, and since $F_{N+n} = h_{N,n} \circ F_N$ converges locally uniformly on $\mathbb{D}$ to $F$, it follows that $(h_{N,n})$ converges locally uniformly on $\mathbb{D}$ to some $H_N \in \mathcal{H}$ with the required property.
\end{remark}

In the proof of Proposition~\ref{prop:Surjective} we write $K_R(z_0)$ for the open disk about $z_0 \in \C$ with radius $R>0$.

\begin{proof}[Proof of Proposition~\ref{prop:Surjective}]
First, we consider the normalized case that $f_n(0)=0$ for each $n \in \N$. Choose $\eps>0$ less than $1/9$. Then, in view of Theorem~\ref{thm:BeniniEtAl} there 
 exists $N \in \mathbb{N}$ such that
\begin{equation}\label{eq:nu-def}
\nu := \prod_{n \ge N+1} |f_n'(0)| > 1 - \varepsilon.
\end{equation}
Remark \ref{rem:factorization} shows that we can write 
$$ F=H_N \circ F_N$$ 
with a nonconstant $H_N \in \mathcal{H}$, $H_N(0)=0$. By the chain rule,
$$ |H_N'(0)|= \prod_{n \ge N+1} |f_n'(0)|=\nu>1-\eps \, .$$
We now make use of the following consequence of the Schwarz lemma,
\begin{equation}\label{eq:classicalineq}
|H_N(z)| \ge |z| \left(1 - \big(1 - |H_N'(0)|\big)\frac{1 + |z|}{1 - |z|}\right), \qquad |z|<|H_N'(0)|, 
\end{equation}
see,  for instance, \cite[Ch.~I, Ex.~1]{garnettBoundedAnalyticFunctions2007} or~\cite[Lem.~2.1]{ferreiraNoteForwardIteration2023}.  
This implies
$$ |H_N(z)| \ge 1-3 \sqrt{\eps} \qquad \text{ for all } |z|=1-\sqrt{\eps} \, .$$
Since $H_n(0)=0$, the argument principle implies that 
the closed curve $H_N(\partial K_{1-\sqrt{\eps}}(0))$ surrounds the origin. By
the previous estimate, it also surrounds each point in $K_{1-3
  \sqrt{\eps}}(0)$.
Thus, again by the argument principle,
\begin{equation}\label{eq:rouche}
H_N\big(K_{1 - \sqrt{\varepsilon}}(0)\big) \supseteq K_{1 - 3\sqrt{\varepsilon}}(0).
\end{equation}
Since $F_N$ is onto, we have $F_N(\mathbb{D}) = \mathbb{D}$ and thus
\begin{equation}\label{eq:ontoresult}
F(\mathbb{D}) = H_N(F_N(\mathbb{D})) = H_N(\mathbb{D}) \supseteq K_{1 - 3\sqrt{\varepsilon}}(0).
\end{equation}
As $\varepsilon > 0$ was arbitrary, it follows that $F: \D \to \D$ is onto.

\medskip

The general case, in which  not necessarily $f_n(0)=0$  for each $n \in \N$, can be reduced to the normalized case as before. Note that passing from $(f_n)$ to $(\widetilde{f}_n)$ via \eqref{eq:normalizing}, and denoting the nonconstant limit of the forward iterates $\widetilde{F}_n$ of $(\widetilde{f}_n)$ by $\widetilde{F}$, we have $\widetilde{F}=T_{F(0)} \circ F \circ T_0$. Since by the previous considerations $\widetilde{F}$ is onto, the same is true of $F$.
\end{proof}

\section{Indestructible Blaschke products -- Proof of Theorem~\ref{thm:IterationOfIBP}}\label{sec:proofIBP}

In this section we prove Theorem~\ref{thm:IterationOfIBP}, which we recall for the reader’s convenience.

\IterationIBP*

\begin{proof}[Proof of Theorem~\ref{thm:IterationOfIBP} \ref{item:IBPlimit} $\Longrightarrow$ \ref{item:IBPsequence}]
Choose $N \in \N$. Then by Remark~\ref{rem:factorization}, $B=H_N \circ B_N$  for some $H_N \in \mathcal{H}$. 
Theorem~1.2 in~\cite{krausCompositionDecompositionIndestructible2013} shows that $H_N$ and  $B_N=b_N \circ (b_{N-1} \circ \cdots \circ b_1)$ are indestructible Blaschke products, and then the same is true for $b_N$.  
\end{proof}

To prove the implication \ref{item:IBPsequence} $\Longrightarrow$ \ref{item:IBPlimit} in Theorem~\ref{thm:IterationOfIBP}, we use the following characterization of indestructible Blaschke products.
We denote by $f^{-1}(a)$ the multiset of zeros of $f-a$. If $a=f(c)$ for some $c \in \D$, then $f^{-1}(a)=Z_f(c)$.

\begin{theoremliterature}[McLaughlin~{\cite[Theorem~1]{mclaughlinExceptionalSetsInner1972}}]\label{thm:CharacterizationOfIBP}
Let $f \in \mathcal{H}$ be an inner function such that
\[
f(z) = f(0) + \widehat{f}(k) z^k + \widehat{f}(k+1) z^{k+1} + \dots \qquad (\widehat{f}(k) \neq 0).
\]
Then $f$ is an indestructible Blaschke product if and only if the following two conditions hold:
\begin{subequations}
\begin{align}
\left\vert \frac{f(0)-a}{1-\overline{a} f(0)} \right\vert &= \prod_{z \in f^{-1}(a)} |z|, \qquad a \in \mathbb{D} \setminus \{ f(0) \}, \label{eq:MLConditionForNonzeros} \\
  \frac{|\widehat{f}(k)|}{1 - |f(0)|^2} &= \prod_{z \in Z_f(0) \setminus \{0\}} |z|. \label{eq:MLConditionForZeros}
\end{align}
\end{subequations}
\end{theoremliterature}

\begin{remark}In~\cite[Theorem 1]{mclaughlinExceptionalSetsInner1972}, $f$ was assumed to be a Blaschke product; it was observed in~\cite[Remark~2.2]{krausCompositionDecompositionIndestructible2013} that this assumption can be weakened to $f$ being an inner function. This refinement is essential for establishing the implication~\ref{item:IBPsequence}~\(\Longrightarrow\)~\ref{item:IBPlimit} of Theorem~\ref{thm:IterationOfIBP}.
\end{remark}

\begin{proof}[Proof of Theorem~\ref{thm:IterationOfIBP}~\ref{item:IBPsequence}~$\Longrightarrow$~\ref{item:IBPlimit}]
Since each $b_n$ is indestructible, Theorem~1.2 in~\cite{krausCompositionDecompositionIndestructible2013} implies that each
\[
B_n := b_n \circ \dots \circ b_1
\]
is an indestructible Blaschke product, and hence satisfies~\eqref{eq:MLConditionForNonzeros} and~\eqref{eq:MLConditionForZeros}.  
By Theorem~1.1 in~\cite{ferreiraNoteForwardIteration2023}, the limit function $B$ is an inner function, so Theorem~\ref{thm:CharacterizationOfIBP} applies to $f = B$.

\medskip

\noindent \emph{Step~1: Verification of~\eqref{eq:MLConditionForNonzeros}.}
Let $a \in \mathbb{D} \setminus \{ B(0) \}$.  
Since indestructible Blaschke products are surjective, each $B_n: \D \to \D$ is onto, and  Proposition~\ref{prop:Surjective}  implies that $B: \D \to \D$ is onto.  
Hence, there exists $c \in \mathbb{D}$ with $B(c) = a$. Since $B_n(c)\not=B_n(0)$  for every $n$ sufficiently large, the only if part of Theorem \ref{thm:CharacterizationOfIBP} shows
$$ \prod_{z \in Z_{B_n}(c)} |z|
 = \left| \frac{B_n(0) - B_n(c)}{1 - \overline{B_n(0)} B_n(c)} \right| . $$
Applying Lemma~\ref{lem:ZeroSetInclusion}, we obtain
\[
\prod_{z \in Z_B(c)} |z|
 = \lim_{n \to \infty} \prod_{z \in Z_{B_n}(c)} |z|
 = \lim_{n \to \infty} \left| \frac{B_n(0) - B_n(c)}{1 - \overline{B_n(0)} B_n(c)} \right|
 = \left| \frac{B(0) - a}{1 - \overline{B(0)} a} \right|.
\]

\medskip

\noindent\emph{Step~2: Verification of~\eqref{eq:MLConditionForZeros}.}

Let $K \ge 1$ denote the order of the zero of $B-B(0)$ at $0$. Then  Lemma~\ref{lem:OrderOfZeroEventuallyConstant} shows that $B_n-B_n(0)$ has a zero of order $K$ at $0$ for sufficiently large $n$.
Combining this with \eqref{eq:ProductLimit}  and the validity of~\eqref{eq:MLConditionForZeros} for each $B_n$,  $n$  sufficiently large, we find
\[
\frac{|\widehat{B}(K)|}{1 - |B(0)|^2}
 = \lim_{n \to \infty} \frac{|\widehat{B_n}(K)|}{1 - |B_n(0)|^2}
 = \lim_{n \to \infty} \prod_{z \in Z_{B_n}(0) \setminus \{0\}} |z|
 = \prod_{z \in Z_B(0) \setminus \{0\}} |z|.
\]

Hence $B$ satisfies both conditions of Theorem~\ref{thm:CharacterizationOfIBP} and is therefore an indestructible Blaschke product.
\end{proof}

	\section{Maximal Blaschke products -- Proof of Theorems~\ref{thm:IterationOfMBP} and \ref{thm:MBPdecomposition} }\label{sec:proofMBP}


We begin with a brief discussion of maximal Blaschke products and some of their basic properties, which will be needed later. For further details, see, for instance, \cite{krausCriticalPointsGauss2013,krausMaximalBlaschkeProducts2013}.

\medskip

Let $f : \D \to \C$ be holomorphic. We denote by $\mathscr{C}_f$ the \emph{critical set} of~$f$, that is, the multiset of its critical points counted with multiplicities. For $f,g \in \mathcal{H}$, we write $\mathscr{C}_f \subseteq \mathscr{C}_g$ if each critical point of~$f$ is also a critical point of~$g$ with at least the same multiplicity. Postcomposition with automorphisms of the disk preserves critical sets, that is,
\[
  \mathscr{C}_{T \circ f} = \mathscr{C}_f, \qquad T \in \Aut(\D).
\]

Let $\mathscr{C} \subseteq \D$ be the critical set of a nonconstant function in~$\mathcal{H}$. A function $F \in \mathcal{H}$ is said to be \emph{maximal for~$\mathscr{C}$} if $\mathscr{C}_F = \mathscr{C}$ and the Schwarz--Pick type inequality
\begin{equation} \label{eq:SchwarzPickMBP}
  \frac{|f'(z)|}{1 - |f(z)|^2} \le \frac{|F'(z)|}{1 - |F(z)|^2}, \qquad z \in \D,
\end{equation}
holds for all $f \in \mathcal{H}$ whose critical set contains~$\mathscr{C}$.

\medskip

The following properties are fundamental.

\begin{enumerate}[label=(P\arabic*), ref=P\arabic*]
\item \textit{Conformal invariance}\label{item:MBP1}\\
If $F \in \mathcal{H}$ is maximal for~$\mathscr{C}$ and $T \in \Aut(\D)$, then $T \circ F$ is also maximal for~$\mathscr{C}$.

\item \textit{Uniqueness up to automorphisms}\label{item:MBP2}\\
If $F, F_* \in \mathcal{H}$ are maximal for~$\mathscr{C}$, then $F_* = T \circ F$ for some $T \in \Aut(\D)$; see~\cite[Theorem~2.5]{krausMaximalBlaschkeProducts2013}.

\item \textit{Existence}\label{item:MaximalExistence}\\
A maximal function for~$\mathscr{C}$ exists; see Heins~\cite[Theorem~13.1]{Heins1962} and also~\cite[Theorem~2.2(a)]{krausMaximalBlaschkeProducts2013}.

\item \textit{Finite critical sets and finite Blaschke products} \label{item:FiniteCriticalSet}\\
If $\mathscr{C}$ is finite with $N$ points, then every maximal function for~$\mathscr{C}$ is a finite Blaschke product of degree $N + 1$. Conversely, every finite Blaschke product is maximal; see Nehari~\cite{Nehari1947} and Heins~\cite{Heins1962}.

\item \textit{Maximal functions are Blaschke products} \label{item:kraus}\\
Every maximal function for $\mathcal{C}$ is a Blaschke product; see \cite[Theorem~1.2]{krausCriticalSetsBounded2013}.
\end{enumerate}

In summary, if $\mathscr{C}$ is the critical set of a nonconstant holomorphic self-map of~$\D$, then there exists~a Blaschke product $B$ whose critical set is precisely $\mathscr{C}$ and such that
\begin{equation} \label{eq:SchwarzPickMBPDef}
  \frac{|f'(z)|}{1 - |f(z)|^2} \le \frac{|B'(z)|}{1 - |B(z)|^2}, \qquad z \in \D,
\end{equation}
for every $f \in \mathcal{H}$ satisfying $\mathscr{C}_f \supseteq \mathscr{C}$.  
Such a Blaschke product is unique up to postcomposition with an automorphism of~$\D$ and is called \textbf{maximal Blaschke product for~$\mathscrbf{C}$}. In particular, every maximal Blaschke product is maximal for its own critical set.

\medskip

An immediate consequence of property~\hyperref[item:FiniteCriticalSet]{(P1)}
and \hyperref[item:kraus]{(P5)} is that every maximal Blaschke product is indestructible.
Moreover, property~\hyperref[item:FiniteCriticalSet]{(P4)} characterizes the finite Blaschke products as precisely those maximal Blaschke products whose set of critical points is finite.
\medskip


For our proof of Theorem~\ref{thm:MBPdecomposition}, the defining inequality~\eqref{eq:SchwarzPickMBPDef} for a maximal Blaschke product is not, by itself, sufficient. Our approach is based on the  insight that the quantity
\[
  \frac{|f'(z)|}{1 - |f(z)|^{2}}
\]
appearing in~\eqref{eq:SchwarzPickMBPDef} may be replaced by any conformal pseudometric $\lambda(z)\,|dz|$ on~$\D$ whose curvature satisfies an appropriate upper bound and whose zero set contains the critical set of $B$. This observation, originating in Ahlfors’ extension of the Schwarz--Pick lemma~\cite{Ahlfors1938} and developed further by Nehari \cite{Nehari1947} and Heins~\cite{Heins1962}, is the source of the additional flexibility that enables the arguments to follow. The key ingredient in the proof of Theorem~\ref{thm:MBPdecomposition} is the following extended maximality property of maximal Blaschke products. Recall that $\lambda(z)\,|dz|$ is said to be a conformal pseudometric on~$\D$ with zero set
\[
\mathscr{C} = \{\underbrace{z_1,\dots,z_1}_{m_1\text{-times}}, \underbrace{z_2,\dots,z_2}_{m_2\text{-times}}, \dots\}
\]
and curvature $\kappa_\lambda \le -4$, if
\begin{enumerate}
\item $\lambda$ is a nonnegative and twice continuously differentiable function on $\D \setminus \mathscr{C}$,  
\item $\displaystyle \limsup_{z \to z_j} \frac{\lambda(z)}{|z - z_j|^{m_j}} < \infty$ for each $j \in \mathbb{N}$,  
\item and
  $$ \kappa_\lambda(z) := -\frac{\Delta (\log \lambda)(z)}{\lambda(z)^2} \le -4$$ for $z \in \D \setminus \mathscr{C}$.
\end{enumerate}

\begin{theoremliterature}[{\cite[Cor.~1.5]{krausCriticalSetsBounded2013}} and {\cite[Th.~1.1]{krausMaximalBlaschkeProducts2013}}]
\label{th:NehariSchwarz}
Let $\mathscr{C}$ be the critical set of a nonconstant function in $\mathcal{H}$ and $B$  a maximal Blaschke product for $\mathscr{C}$.  
Suppose that $\lambda(z)\,|dz|$ is a conformal pseudometric on $\D$ with zero set $\mathscr{C}$ and curvature $\kappa_\lambda \le -4$.
Then
\[
\lambda(z) \le \frac{|B'(z)|}{1 - |B(z)|^2}, \quad z \in \D.
\]
\end{theoremliterature}

\begin{remark} \label{rem:ConstantCurvature}
   If $f \in \mathcal{H}$, then the quantity
\[
\frac{|f'(z)|}{1 - |f(z)|^2}\, |dz|
\]
defines a conformal pseudometric on $\D$ with zero set $\mathscr{C}_f$ and constant curvature $-4$, that is,
\[
\kappa_\lambda(z) = -4 \quad \text{for } z \in \D \setminus \mathscr{C}_f.
\]
\end{remark}
  
\begin{proof}[Proof of Theorem \ref{thm:MBPdecomposition}]
In view of Remark \ref{rem:MBPdecomposition}, we only need to prove that if $A=A_1 \circ A_2$ is a maximal Blaschke product, then so is $A_2$. Suppose $A=A_1 \circ A_2$ is a maximal Blaschke product.
Obviously, $A_2$ is a nonconstant function in $\mathcal{H}$.
Therefore, in view of property \hyperref[item:MaximalExistence]{(P3)} and property \hyperref[item:kraus]{(P5)}, there is~a maximal Blaschke product $B$ for $\mathscr{C}_{A_2}$. In particular,
\begin{equation}\label{eq:inequal_1}
\frac{|A_2'(z)|}{1-|A_2(z)|^2} \le  \frac{|B'(z)|}{1-|B(z)|^2}, \quad z \in
\D \, .
  \end{equation}
 For convenience, we use the short-hand notation
  \begin{equation*}
\lambda_A(z)=\frac{|A'(z)|}{1-|A(z)|^2}, \quad
\lambda_{A_2}=\frac{|A_2'(z)|}{1-|A_2(z)|^2} \quad \text{and} \quad \lambda_{B}(z)= \frac{|B'(z)|}{1-|B(z)|^2} \, .
    \end{equation*}
 Since $\mathscr{C}_{A_2}=\mathscr{C}_{B}$, which means that $A_2'$ and $B'$ have the same zeros (counting multiplicities),  the function $z
\mapsto \lambda_{B}(z)/\lambda_{A_2}(z)$, which is originally defined  on $\D \setminus \mathscr{C}_{A_2}$,  has a continuous extension to~$\D$, and
 \eqref{eq:inequal_1} shows
\begin{equation}\label{eq:inequal_2}
  1 \le  \frac{\lambda_{B}(z)}{\lambda_{A_2}(z)}< \infty \quad \text{for } z \in \D\,.
\end{equation}
In order to  establish that $A_2$ is a maximal Blaschke product we aim to prove the reverse inequality $\lambda_B \le \lambda_{A_2}$.  The main idea is to introduce the auxiliary function
$$\sigma(z):=\lambda_A(z) \cdot \frac{\lambda_{B}(z)}{\lambda_{A_2}(z)}, \qquad z \in \D \, . $$
Then, in view of \eqref{eq:inequal_2}, $\sigma(z) \, |dz|$ is a conformal pseudometric on $\D$ with zero set $\mathscr{C}_A$. We proceed to show that $\kappa_\sigma \le -4$ on $\D \setminus \mathscr{C}_A$.  
Taking into account the definition of curvature and Remark \ref{rem:ConstantCurvature}, we have
  $$ \Delta (\log \sigma) =\Delta (\log \lambda_A)+\Delta (\log
  \lambda_B)-\Delta (\log \lambda_{A_2})=4 \lambda_A^2+4
  \lambda_B^2-4\lambda_{A_2}^2.$$
  Therefore
\begin{equation} \label{eq:sigmaCurvature}
\begin{split}
\kappa_{\sigma}(z)=-\frac{\Delta (\log \sigma)(z)}{\sigma(z)^2}&=-4\left(  \left(\frac{\lambda_{A_2}(z)}{\lambda_{B}(z)}
  \right)^2 +\left(\frac{\lambda_{A_2}(z)}{\lambda_{A}(z)} \right)^2-
                                                               \left(\frac{\lambda_{A_2}(z)}{\lambda_{A}(z)}
                                                               \right)^2
                                                               \left(\frac{\lambda_{A_2}(z)}{\lambda_{B}(z)}
                                                               \right)^2
                                                               \right)\\&=-4
  \left(  \left(\frac{\lambda_{A_2}(z)}{\lambda_{B}(z)}\right)^2 +
  \left(\frac{\lambda_{A_2}(z)}{\lambda_{A}(z)} \right)^2  \left(1-
  \left(\frac{\lambda_{A_2}(z)}{\lambda_{B}(z)}\right)^2 \right)  \right) \, .
\end{split}
\end{equation}
In order to estimate $\kappa_\sigma$ on $\D \setminus \mathscr{C}_A$ from above, we notice that
\begin{equation}\label{eq:curv_2}
\frac{\lambda_{A_2}(z)}{\lambda_A(z)}\ge 1 \quad \text{for } z \in \D
\backslash \mathscr{C}_A\,.
\end{equation}
This follows at once from the Schwarz--Pick lemma applied to  $A_1\in \mathcal{H}$, which leads to
$$ \lambda_A(z)=\frac{|A'(z)|}{1-|A(z)|^2}=\frac{|A_1'(A_2(z))| \cdot |A_2'(z)|}{1-|A_1(A_2(z))|^2}\le
\frac{ |A_2'(z)|}{1-|A_2(z)|^2}= \lambda_{A_2}(z)\, , \quad z \in \D\, . $$
Combining \eqref{eq:inequal_2} in the form $1-(\lambda_{A_2}/\lambda_B)^2 \ge 0$ and \eqref{eq:curv_2}, we see that the right-hand side of \eqref{eq:sigmaCurvature} is bounded above by $-4$. We are therefore in a position to apply Theorem \ref{th:NehariSchwarz} to the conformal pseudometric $\sigma(z) \, |dz|$, which leads to 
  \begin{equation*}
\lambda_A(z) \cdot \frac{\lambda_B(z)}{\lambda_{A_2}(z)}=\sigma(z) \le \frac{|A'(z)|}{1-|A(z)|^2}=\lambda_A(z), \quad z \in \D. 
\end{equation*}
Thus $\lambda_{B}(z) \le \lambda_{A_2}(z) $ for $z \in \D$. 
Combining this with \eqref{eq:inequal_1} gives
\begin{equation*}
\frac{|A_2'(z)|}{1-|A_2(z)|^2}=\frac{|B'(z)|}{1-|B(z)|^2} \quad
\text{for } z \in \D\, \, . 
\end{equation*}
Therefore, $A_2$ is maximal for $\mathcal{C}_{A_2}$, and hence it is a maximal Blaschke product by  property \hyperref[item:kraus]{(P5)}.
\end{proof}

\begin{proof}[Proof of Theorem \ref{thm:IterationOfMBP}]

Having Theorem \ref{thm:MBPdecomposition} at hand, the proof of the implication \ref{item:MBPlimit}~$\implies$~\ref{item:MBPsequence} follows by the same argument as in the case of indestructible Blaschke products, see the proof of Theorem \ref{thm:IterationOfIBP} \ref{item:IBPlimit} $\Longrightarrow$ \ref{item:IBPsequence}.

\medskip

	It remains to prove \ref{item:MBPsequence} $\implies$ \ref{item:MBPlimit}. We denote by~$\mathscr{C}=\mathscr{C}_B$ the set of critical points of~$B$. We wish to show that~$B$ is a maximal Blaschke product for~$\mathscr{C}$. In view of our preliminary discussion, see  \hyperref[item:MaximalExistence]{(P3)} and  \hyperref[item:kraus]{(P5)},  there exists a maximal Blaschke product~$M$ for $\mathscr{C}$. In particular,
	\begin{equation}
		\frac{|B'(z)|}{1-|B(z)|^2}\leq	\frac{|M'(z)|}{1-|M(z)|^2},\quad z\in\D.
	\end{equation}
	Further, denote~$\mathscr{C}_n=\mathscr{C}_{B_n}$ for all $n\in\N$. In view of Theorem \ref{thm:MBPdecomposition} \ref{item:MBPComposition}~$\Longrightarrow$~\ref{item:MBPComposite},  the function $B_n$ is a maximal Blaschke product for~$\mathscr{C}_n$, and clearly $\mathscr{C}_n\subseteq\mathscr{C}$ for all $n\in\N$. Together, this implies by \eqref{eq:SchwarzPickMBPDef}
	\begin{equation}
		\frac{|M'(z)|}{1-|M(z)|^2}\leq	\frac{|B_n'(z)|}{1-|B_n(z)|^2},\quad z\in\D
	\end{equation}
	for every $n\in\N$.  \hide{Moreover, in view of $\mathscr{C}_{n}\subseteq \mathscr{C}_{n+1}$, again we see from \eqref{eq:SchwarzPickMBP} that
$$        	\frac{|B_{n+1}'(z)|}{1-|B_{n+1}(z)|^2}\le 	\frac{|B_n'(z)|}{1-|B_n(z)|^2}, \quad z \in \D\,  .$$
The last two displayed formulas imply that the limit
$$\lim_{n\to\infty}\frac{|B_n'(z)|}{1-|B_n(z)|^2} \, $$
exists, and} Since $B_n \to B$ locally uniformly in $\D$ and hence $B_n' \to B'$, we deduce that
	\begin{equation}
		\frac{|B'(z)|}{1-|B(z)|^2}\leq	\frac{|M'(z)|}{1-|M(z)|^2}\leq\lim_{n\to\infty}\frac{|B_n'(z)|}{1-|B_n(z)|^2}=\frac{|B'(z)|}{1-|B(z)|^2},\quad z\in\D \, . 
              \end{equation}
Hence
$B$ is a maximal Blaschke product by properties \hyperref[item:MBP1]{(P1)}, \hyperref[item:MBP2]{(P2)} and \hyperref[item:kraus]{(P5)}.
This concludes the proof of Theorem~\ref{thm:IterationOfMBP}.
        \end{proof}

\section{Proof of Theorem~\ref{thm:BeniniEtAl}} \label{sec:TheoremA}

\noindent
By hypothesis,  the forward iterates $F_n$ of $(f_n)$ converge to a limit $F \in \mathcal{H}$.
Define
\[
\lambda_n(z) := \frac{|F_n'(z)|}{1 - |F_n(z)|^2}, \qquad z \in \D.
\]
We first  assume $f_n(0) = 0$ for each $n \in \N$.

\medskip

By the Schwarz--Pick lemma applied to $f_{n+1} \in \mathcal{H}$, we have
\[
  \lambda_{n+1}(z) = \frac{|f_{n+1}'(F_n(z))|}{1 - |f_{n+1}(F_n(z))|^2} \, |F_n'(z)| \le \frac{|F_n'(z)|}{1 - |F_n(z)|^2}=
  \lambda_n(z), \quad z \in \D.
\]
Hence $(\lambda_n)$ is a pointwise nonincreasing sequence of nonnegative functions, and thus converges to
\[
 \frac{|F'|}{1 - |F|^2},
\]
where  either $F$ is identically zero or $F$ has a zero of order $K \in \N$ at $z=0$.

\medskip

In the latter case, since $F$ is the limit of the forward iterates of $(f_n)$, there exists $N \in \N$ such that $f_n'(0) \neq 0$ for all $n \ge N$, and we can write
\[
F = H_N \circ F_N
\]
for some nonconstant $H_N \in \mathcal{H}$ (see Remark~\ref{rem:factorization}), satisfying
\[
\prod_{j=N+1}^{\infty} |f_j'(0)| = |H_N'(0)| \neq 0.
\]
Consequently,
\begin{equation} \label{eq:BlaschkeCondition}
\sum_{j=1}^{\infty} (1 - |f_j'(0)|) < \infty.
\end{equation}

\medskip

In the first case, $F$ is identically zero, condition~\eqref{eq:BlaschkeCondition} cannot hold: it would imply that for some $M \in \N$ there exists $\varepsilon > 0$ such that
\[
\prod_{j=M+1}^{\infty} |f_j'(0)| \ge \varepsilon.
\]

As in Remark \ref{rem:factorization}, setting  $h_{M,n}:=f_{M+n} \circ \cdots \circ f_{M+1} \in \mathcal{H}$, we would have $F_{M+n}=h_{M,n} \circ F_M$ with $|h_{M,n}'(0)|\ge \eps$ for each $n \ge 1$.
This allows us to write $F = H_M \circ F_M$ with $H_M \in \mathcal{H}$ and $|H_M'(0)| \ge \varepsilon$, contradicting $F_M \not\equiv 0$.

\medskip

Finally, we indicate how to deal with the general case when $f_n(0)$ is not necessarily $0$. Fix $a \in \D$ and consider the sequence
  $(\tilde{f}_n)$ defined by \eqref{eq:Normalization} and its forward iterates $\tilde{F}_n$. Note that $\tilde{f}_n(0)=0$. Then $F$ is nonconstant if and only if $(\tilde{F}_n)$ converges to a nonconstant limit $\tilde{F}$. In view of 
$$ \frac{1-D_h\tilde{f}_n(0)}{1-D_h f_n(a)}=\frac{1-D_hf_n(F_{n-1}(a))}{1-D_h f_n(a)},$$
the   estimate  \cite[(3.4)]{GumenyukKourouMouchaRoth2025}
$$ e^{-2 \dD(a,F_{n-1}(a))} \le \frac{1-D_hf_n(F_{n-1}(a))}{1-D_h f_n(a)} \le e^{2 \dD(a,F_{n-1}(a))} , $$
where $\dD$ denotes the hyperbolic distance, 
  and the assumption that $(F_n)$ converges to an element in~$\mathcal{H}$, 
  we see  that the two series $\sum (1-D_h f_n(a))$ and $\sum (1-D_h\tilde{f}_n(0))$ are equisummable.
Hence, the general case is reduced to the special case $f_n(0)=0$, and this completes the proof of Theorem~\ref{thm:BeniniEtAl}.

\medskip

\noindent
\textbf{Acknowledgement.} We are grateful to the referees for their careful reading of the manuscript and insightful suggestions.

\end{document}